\documentclass[11pt]{article}
 \usepackage{amsmath,amssymb,amsthm,amsfonts}
 \usepackage{epstopdf}
 \RequirePackage[dvips]{graphicx}
 \usepackage[dvips]{epsfig}
 \usepackage{color}
\usepackage[usenames,dvipsnames]{pstricks}
\usepackage{tikz}
\usepackage{kotex}
\usepackage{soul}
\usepackage{lineno}

 \def\draw #1 by #2 (#3){
  \vbox to #2{
    \hrule width #1 height 0pt depth 0pt
    \vfill
    \special{picture #3} 
    }
  }

 \def\scaleddraw #1 by #2 (#3 scaled #4){{
  \dimen0=#1 \dimen1=#2
  \divide\dimen0 by 1000 \multiply\dimen0 by #4
  \divide\dimen1 by 1000 \multiply\dimen1 by #4
  \draw \dimen0 by \dimen1 (#3 scaled #4)}
  }

\newtheorem{theorem}{Theorem}[section]
\newtheorem{example}[theorem]{Example}
\newtheorem{problem}[theorem]{Problem}

\newtheorem{defin}[theorem]{Definition}
\newtheorem{lemma}[theorem]{Lemma}
\newtheorem{corollary}[theorem]{Corollary}
\newtheorem{remark}[theorem]{Remark}
\newtheorem{nt}{Note}

 \newtheorem{rule-def}[theorem]{Rule}

\begin{document}
 \newcommand{\la}{\lambda}
 \newcommand{\si}{\sigma}
 \newcommand{\ol}{1-\lambda}
 \newcommand{\be}{\begin{equation}}
 \newcommand{\ee}{\end{equation}}
 \newcommand{\bea}{\begin{eqnarray}}
 \newcommand{\eea}{\end{eqnarray}}

\author{Gi-Sang Cheon\thanks{Applied Algebra and Optimization Research Center, Department of Mathematics, Sungkyunkwan University, Suwon 16419, Republic of Korea}, Ji-Hwan Jung\footnotemark[3], Bumtle Kang\footnotemark[2],  Hana Kim\footnotemark[2],\\
 Suh-Ryung Kim\thanks{Department of Mathematics Education, Seoul National University, Seoul 08826, Republic of Korea},
 Sergey Kitaev\thanks{Department of Mathematics and Statistics, University of Strathclyde, 26 Richmond Street, Glasgow,
  G1 1XH, United Kingdom} \ and \ Seyed Ahmad Mojallal\thanks{GERAD and HEC Montr\'{e}al, Montr\'{e}al, H3T 1J4, Canada}
  \\
{\footnotesize gscheon@skku.edu},\,{\footnotesize
jihwanjung@snu.ac.kr},\,{\footnotesize lokbt1@skku.edu},\,{\footnotesize hakkai14@skku.edu},\,\\
{\footnotesize srkim@snu.ac.kr},\,{\footnotesize sergey.kitaev@cis.strath.ac.uk},\,{\footnotesize
Ahmad\_mojalal@yahoo.com}}

\title{{Counting independent sets in Riordan graphs}\thanks{This work was supported by the National Research Foundation of Korea (NRF) grant funded by the Korean Government (MSIP) (2016R1A5A1008055), the Ministry of Education of Korea (NRF-2019R1I1A1A01044161), and  by the Korea government (MEST) (NRF-2017R1E1A1A03070489).}
\date{}}

\maketitle

 \begin{abstract}
The notion of a Riordan graph was introduced recently, and it is a far-reaching generalization of the well-known Pascal graphs and Toeplitz graphs. However, apart from a certain subclass of Toeplitz graphs, nothing was known on independent sets in Riordan graphs.

In this paper, we give exact enumeration and lower and upper bounds
for the number of independent sets for various classes of Riordan
graphs. Remarkably,  we offer a variety of methods to solve the
problems that range from the structural decomposition theorem to
methods in combinatorics on words. Some of our results are valid for
any graph.

 \bigskip

 \end{abstract}

 \section{Introduction}

 An {\em independent set}, also known as a {\em stable set}, of a graph is a set of vertices in the graph, no two of which are
 adjacent.
 A {\em maximum independent set} for a given graph $G$ is an independent set of largest possible size. This size is called the {\it independence number} of $G$, and is denoted by $\alpha(G)$, or simply by $\alpha$ if $G$ is clear from the context.
 The problem of finding such a set is called the {\em maximum independent set problem}, and it is an NP-hard optimization problem (see \cite{Butenko} for a survey paper).
 Independent sets in graphs are studied extensively in the literature for various classes of graphs (see \cite{S} for a short survey paper).

 In this paper, we study the number of independent sets in ``Riordan
graphs" which have been introduced recently in~\cite{CJKM,CJKM2}. Riordan graphs have a number of interesting properties
\cite{CJKM} and applications such as creating computer
networks with certain desirable features and designing algorithms to compute values of
graph invariants \cite{DQ}.

A {\em
Riordan matrix} $L=[\ell_{ij}]_{i,j\ge0}$ generated by two formal
power series $g=\sum_{n=0}^\infty g_nz^n$ and $f=\sum_{n=1}^\infty
f_nz^n$ in ${\mathbb Z}[[z]]$ is denoted as $(g,f)$ and defined as
an infinite lower triangular matrix whose $j$-th column generating
function is $gf^j$, {\it i.e.} $\ell_{ij}=[z^i]gf^j$ where
$[z^k]\sum_{n\ge0}a_nz^n=a_k$. If $g_0\ne0$ and $f_1\ne
 0$ then the Riordan matrix is called {\it proper}. For example, $P=\left({1\over 1-z},{z\over 1-z}\right)$ is proper
and $T=\left(z+z^2+z^3,z\right)$ is not proper.

A simple graph $G$ of order $n$ is said to be a {\em Riordan graph} if the vertices of $G$ can be labelled as $1,2,\ldots,n$ so that its adjacency matrix $A(G)$ whose $i$th row corresponds to the vertex $i$ can be expressed as
    \begin{equation} A(G) \equiv (zg,f)_n + (zg,f)_n^T \pmod{2} \label{riordandef}\end{equation}
    for some generating functions $g$ and $f$ over $\mathbb Z$ where $(zg,f)_n$ is the $n\times n$ leading principle matrix of the Riordan matrix
    $(zg,f)$. The Riordan graph $G$ on $n$ vertices with
the adjacency matrix $A(G)$ given by (\ref{riordandef}) is denoted
as $G=G_n(g,f)$. If we let $A(G)=(a_{i,j})_{1\le i,j\le n}$, then,
for $i\ge j$,
    \begin{align*}
a_{i,j}=a_{j,i}\equiv [z^{i-2}]gf^{j-1} \pmod{2}
\end{align*}
by (\ref{riordandef}). In particular, if $[z^0]g \equiv [z^1]f
\equiv 1 \pmod{2}$, then the graph $G_n(g,f)$ is called {\it
proper}.

\begin{table}
 \begin{center}
 \begin{tabular}{c|c|c|c|c|c|c|c|c|c|c|c|c}
 $n$ & 1 & 2 & 3 & 4 & 5 & 6 & 7 & 8 & 9 & 10 &11 & 12 \\
 \hline
 $i(PG_n)$ & 2 & 3 & 4 & 6 & 7 & 12 & 15 & 23 & 24 & 46 & 60 & 98\\
$i(MG_n)$ & 2 & 3 & 4 & 7 & 9 & 13 & 17 & 26 & 29 & 48 & 55 & 95 \\
$i(CG_n)$ & 2 & 3 & 4 & 7 & 8 & 14 & 21 & 35 & 36 & 60 & 81 & 134 \\
 \end{tabular}
 \caption{The number of independent sets for $1\leq n\leq 12$}\label{tab-data}
 \end{center}
 \end{table}

There are several naturally defined classes/families of Riordan graphs~\cite{CJKM,CJKM2}.
In this paper, we focus on two types of Riordan graphs: the Appell type and the Bell type.
A Riordan graph $G_n(g,z)$ is said to be {\em of the Appell type}. Riordan graphs of the Appell type are also known as {\em Toeplitz graphs}.
Toeplitz graphs have been studied in \cite{DTTVZZ,NP,E,IB,Ghorban}.

The adjacency matrix $A(G_n)=(t_{i,j})_{1\le i,j\le n}$ of a
Toeplitz graph $G_n$ with $n$ vertices is the $(0,1)$-symmetric
matrix such that $ij\in E(G_n)$ with $j < i$ if and only if
$t_{i-j+s,s}=t_{s,i-j+s}=1$ for each $s=1,\ldots,n-i+j$. Thus, one
can see that a graph $G$ is a Toeplitz graph $G_n(g,z)$ with $n$
vertices for $g=z^{t_1-1}+z^{t_2-1}+\cdots+z^{t_k-1}$ if and only if
there exist positive integers $t_1,t_2,\ldots,t_k\le n-1$ such that
$E(G)=\{ij\;|\;|i-j|=t_s,\;s=1,\ldots,k\}$. In this context, we can
represent a Toeplitz graph $G_n(g,z)$ for
$g=z^{t_1-1}+z^{t_2-1}+\cdots+z^{t_k-1}$ as $T_n\langle t_1,
t_2,\ldots, t_k\rangle$.

A Riordan graph $G_n(g,zg)$ is said to be {\em of the Bell type}.
Moreover, a Riordan graph $G_n(g,zg)$ is called the {\em Pascal graph} and denoted by $PG_n$, the {\em Catalan graph}  and denoted by $CG_n$, and the {\em Motzkin graph} and denoted by $MG_n$ if $g=\frac{1}{1-z}$, $g=\frac{1-\sqrt{1-4z}}{2z}$, and $g=\frac{1-z-\sqrt{1-2z-3z^2}}{2z^2}$, respectively.

See Table~\ref{tab-data} for our experimental computations for the numbers of independent sets in Pascal graphs, Motzkin graphs, and Catalan graphs with small numbers of vertices. These data can be used to measure the sharpness of the upper bounds obtained in this paper.

The following theorem regarding the adjacency matrices of Riordan graphs plays a key role throughout the paper.
Given a positive integer $n$, we let  \[N_o:=\{2i-1\;|\;1\le i\le \lceil n/2\rceil\} \quad \mbox{and} \quad N_e:=\{2i\;|\;1\le i\le \lfloor n/2\rfloor\}.\]
For a given graph $G$ with $n$ vertices labeled by $1,\ldots,n$, we denote by $\langle V_o\rangle $ and $\langle V_e\rangle $ the subgraphs of $G$ induced by $N_o$ and $N_e$, respectively.

 \begin{theorem}[Riordan Graph Decomposition, \cite{CJKM}]\label{e:th}  Let $G_n=G_n(g,f)$ be a Riordan graph  with $[z^1]f=1$. Then,
\begin{itemize}
\item[{\rm(i)}] the induced subgraph $\left<V_o\right>$ is a Riordan graph of order
${\lceil n/2\rceil}$ given by $G_{\lceil n/2\rceil}(g^{\prime}(\sqrt{z}) ,f(z))$;
\item[{\rm(ii)}] the induced subgraph $\left<V_e\right>$ is a Riordan graph of order
${\lfloor n/2\rfloor }$ given by $G_{\lfloor n/2\rfloor
}\left(\left(\frac{gf}{z}\right)^{\prime}(\sqrt{z}),f(z)\right)$;
\item[{\rm(iii)}] For the adjacency matrix $A(G_{n})$, suppose that
\begin{eqnarray}\label{e:bm}
A(G_{n})=P^{T}\left[
\begin{array}{cc}
X & B \\
B^{T} & Y
\end{array}
\right]P
\end{eqnarray}
\end{itemize}
where  $P=\left[e_{1}\;|\; e_{3}\;|\; \cdots\;|\; e_{2\lceil
n/2\rceil -1}\;|\; e_{2}\;|\; e_{4} \;|\; \cdots \;|\; e_{2\lfloor
n/2\rfloor }\right] ^{T}$ is the $n\times n$ permutation matrix and
$e_{i}$ is the elementary column vector with the $i$-th entry equal
$1$ and the other entries equal~$0$.
  Then, the matrix $B$ representing the adjacency of a vertex in $V_o$ and a vertex in $V_e$ can be expressed as the sum of two Riordan matrices as follows:
\begin{align*}
B\equiv(z\cdot(gf)^{\prime }(\sqrt{z}),f(z))_{\lceil n/2\rceil
\times\lfloor n/2\rfloor}+((zg)^{\prime}(\sqrt{z}),f(z)) _{\lfloor
n/2\rfloor\times\lceil n/2\rceil }^{T}\pmod 2.
\end{align*}
\end{theorem}

 Let $G_n=G_{n}(g,f)$ be a proper Riordan graph. If $\left< V_{o}\right> \cong G_{\lceil n/2\rceil}(g,f)$ and $\left<
V_{e}\right>$ is a null graph, then $G_n$ is said to be {\em isomorphically odd decomposable} and is abbreviated as
{\em io-decomposable}.

In the rest of this paper, we let $[n]:=\{1,2,\ldots,n\}$.

This paper is organized as follows. In Section~\ref{chordalToeplitzexact}, we study the number of independent sets in a Toeplitz graph which is a Riordan graphs of the Appell type.
We give a lower bound and an upper bound for the number of independent sets in a Toeplitz graph and provide the exact number of independent sets in a chordal Toeplitz graph in terms of the Fibonacci numbers. In Section~\ref{upperBoundsSec}, we study the number of independent sets in a Riordan graph of the Bell type. We give  upper bounds for various Riordan graphs of the Bell type in terms of the Pell numbers. We also find the independence number and the number of maximal independent sets in an io-decomposable Riodan graph of the Bell type. Then, we give a lower bound for the number of independent sets in an io-decomposable Riordan graph through a decomposition of graphs.
 Finally, in Section~\ref{further-research-sec} we discuss directions of further research.

\section{Independent sets in Toeplitz graphs}\label{chordalToeplitzexact}

 In \cite{Kit}, independent sets on {\em
path-schemes} are considered. As a matter of fact, the class of path-schemes is
precisely the class of Toeplitz graphs.
Explicit enumeration of independent sets in terms of generating
functions is obtained in \cite{Kit} for a subclass of Toeplitz
graphs defined by the notion of a {\em well-based sequence} (see
\cite{V} for enumerative properties of such sequences). The
combinatorics on words approach used in the enumeration relies on
the methods developed in \cite{GO} to count pattern-avoiding words.
In fact, a result in \cite{Kit} gives an upper bound on the number
of independent sets for any Toeplitz graph, which was not mentioned
in \cite{Kit}.

 Suppose that $k\geq 2$ and $\mathcal{A}=\{A_1,\ldots,A_k\}$ is a set of words of the form
$A_i=1\underbrace{0\ldots 0}_{a_i-1}1$, where $a_i\geq 1$, and
$a_i<a_j$ if $i<j$. Moreover, we assume that for any $i>1$ and
$A_i\in\mathcal{A}$, if we replace any number of $0$'s in $A_i$ by $1$'s,
then we obtain a word $A'_i$ that contains the word
$A_j\in\mathcal{A}$ as a factor for some $j<i$. In this case, we
call $\mathcal{A}$ a {\em well-based\index{well-based set} set}, and
we call the sequence of $a_i$'s associated with $\mathcal{A}$ a {\em
well-based sequence}. It is known that any well-based set must contain the word $11$~\cite{Kit}.
\begin{theorem}(\cite{Kit}) Let $t_1,t_2,\ldots,t_k$ be a well-based sequence with $t_1 = 1$ where $\{t_1,t_2,\ldots,t_k\}$ is a subset of $[n]$. Let $c(x) = 1+ \sum_{i=1}^k x^{t_i}$. Then the generating function for the number of independent sets in $T_n\left<t_1,t_2,\ldots,t_k\right>$ with the vertex set $V=[n]$ is given by\[
G(x) = \frac{c(x)}{(1-x)c(x) - x}.\]\label{thm:kit}\end{theorem}

The above theorem can be restated to give a lower bound for the number of independent sets in any Toeplitz graph.

\begin{theorem}\label{Toeplitz-thm} Let $G_n=T_n\langle t_1, t_2,\ldots, t_k\rangle$ be a Toeplitz graph with $n$ vertices, and $A_n=\{t_1,t_2,\ldots, t_k\}$. Further, let $B_n=\{b_1,b_2,\ldots,b_t\}$ be a minimal, possibly
empty subset of $[n]$ such that $A_n\cap B_n=\emptyset$ and the
elements of $A_n\cup B_n$ form a well-based sequence.
  Finally, let $c(x)=1+\sum_{i=1}^{k}x^{t_i}+\sum_{j=1}^{t}x^{b_j}$. Then,
$$i(G_n)\geq [x^n]\frac{c(x)}{(1-x)c(x)-x},$$ where the equality
holds if and only if $B_n=\emptyset$, that is, the elements of $A_n$ form a well-based sequence.  \end{theorem}
\begin{proof} Let $s_1,\ldots,s_{t+k}$ be an increasing sequence formed by the element in $A_n\cup B_n$ and $G'_n:=T_n\langle s_1,\ldots,s_{t+k}\rangle$. Since $G_n$ is a spanning subgraph of $G'_n$, we have $i(G_n)\geq i(G'_n)$ where the equality holds if and only if $B_n=\emptyset$. Since the elements of $A_n\cup B_n$ form a well-based sequence, Theorem~\ref{thm:kit} can be applied to the Toeplitz graph $G'_n$ to conclude that $$\sum_{i\geq 0}i(G'_n)x^i=\frac{c(x)}{(1-x)c(x)-x}$$ and the desired result follows.\end{proof}

Let $F_k(n)$ be the {\em $n$-th $k$-generalized Fibonacci number} defined
by
\begin{align*}
F_k(1)=\cdots=F_k(k)=1\; \textrm{and}\; F_k(n) = F_k(n-1)+F_k(n-k)\;
\textrm{for}\; n>k.
\end{align*}
In particular, when $k=2$, $F_k(n)$ is the {\em $n$-th Fibonacci number} $F(n)$.

\begin{lemma}[\cite{Kit}]\label{Toeplitz-chordal-3}
$F_k(n+k)$ counts the number of independent sets in $T_n\langle 1,
2,\ldots, k-1\rangle$.
\end{lemma}

 Now, we characterize the Riordan graphs $G_n$ satisfying
$i(G_n)\le F_k(n)$.

 \begin{theorem}
 Let $k\ge3$ and $G_n^{(k)}=G_n(g,f)$ be a Riordan graph such that $[z^i]g\equiv
 1$ for each $i=0,\ldots,k-2$, $[z^1]f\equiv 1$ and $[z^j]f\equiv 0\pmod{2}$ for each $j=2,\ldots,k-1$.
 Then,
 \begin{equation*}
 i(G_n^{(k)}) \le F_k(n+k)\label{equ11}\end{equation*}
 and the equality holds if and only if $G_n^{(k)}=T_n\left<1,2,\ldots,k-1\right>$.
\end{theorem}
\begin{proof}
  It is easy to see that if a graph $H$ is a spanning subgraph of $G$, then $i(G) \le i(H)$ and the equality holds if and only if $G=H$.
For each $k \ge 3$, let $g$ and $f$ be generating functions over $\mathbb{Z}$ such that
\[ [z^i]g\equiv 1\pmod{2} \quad \mbox{for each} \ i=0,\ldots,k-2,\]
\[ [z^1]f\equiv 1 \pmod{2} \ \mbox{and} \ [z^j]f\equiv 0\pmod{2} \quad \mbox{for each} \ j=2,\ldots,k-1.\]
Then, we consider a Riordan graph
$$G_n^{(k)}=G_n(g,f).$$
By definition,
$G_n(\sum_{i=0}^{k-2}z^i,z)=T_n\left<1,2,\ldots,k-1\right>$. By
Lemma~\ref{Toeplitz-chordal-3}, \[i(T_n\left<1,2,\ldots,k-1\right>)
= F_k(n+k).\] As a matter of fact, $T_n\left<1,2,\ldots,k-1\right>$
is a spanning subgraph of $G_n^{(k)}$. Indeed, take two adjacent
vertices $u$ and $v$ in $T_n\left<1,2,\ldots,k-1\right>$. Without
loss of generality, we assume that $u > v$. Then, $1 \le u-v \le
k-1$. To see the adjacency of $u$ and $v$ in $G_n^{(k)}$, we compute
$[z^{u-2}]gf^{v-1}$. By the definitions, \begin{equation}
g=1+z+\cdots+z^{k-2}+O(z^{k-1})\quad \mbox{and} \quad
f=z+O(z^{k}).\label{eq:gandf}\end{equation} By~\eqref{eq:gandf}, the
minimum degree of the nonzero term in $f^{v-1}$ is $v-1$ and the
second minimum degree of the nonzero term in $f^{v-1}$ is at least
$v+k-2$, if it exists.
However, the term $z^{u-v-k}$ does not exist in $g$ since $u-v \le
k-1$. Since $0 \le u-v-1 \le k-2$, we obtain $[z^{u-v-1}]g=1$. Thus,
$[z^{u-2}]gf^{v-1}=1$, which is obtained by multiplying $z^{v-1}$ in
$f^{v-1}$ by $z^{u-v-1}$ in $g$. Therefore, $u$ and $v$ are adjacent
in $G_n^{(k)}$ which completes the proof.
\end{proof}

Even though in our paper we normally obtain lower and/or upper bounds for the number of independent sets in question, in this section, we find the {\em exact} number of independent sets in
Toeplitz graphs related to chordal graphs. A {\em chordal graph} is a simple graph in which every graph cycle of length four and greater has a cycle chord.

\begin{lemma}[\cite{CJKKM}]\label{Toeplitz-chordal-1} For $k\ge1$, let $G_n:=T_n\langle t_1, t_2,\ldots, t_k\rangle$ be a Toeplitz graph of order $n\ge t_k+t_{k-1}+1$. Then, $G_n$ is
chordal if and only if $t_j=jt_1$ for each
$j=1,\ldots,k$.
\end{lemma}

\begin{lemma}[\cite{CJKKM}]\label{Toeplitz-chordal-2} For $k,t\ge1$, let $G_n=T_n\langle t, 2t,\ldots,
kt\rangle$ be a Toeplitz graph of order $n\ge (2k-1)t+1$. Then,
every connected component $H_i$, $i=1,\ldots,t$, of $G_n$ is
isomorphic to the graph $T_{\lfloor(n-i)/t\rfloor+1}\langle 1,
2,\ldots, k\rangle$, and the vertex set of $H_i$ is
\begin{align*}
V(H_i)=\{v\in [n]\;|\; v=i+st,\;s=0,1,\ldots\}.
\end{align*}
\end{lemma}
For each $i$, we can easily check that any $k+1$ consecutive vertices in $T_{\lfloor(n-i)/t\rfloor+1}\langle 1, 2,\ldots, k\rangle$ is a clique, so as $H_i$.

The following theorem follows immediately from Lemmas
\ref{Toeplitz-chordal-1}, \ref{Toeplitz-chordal-2} and~\ref{Toeplitz-chordal-3}.

\begin{theorem}For $k\ge1$, let $G_n=T_n\langle t, 2t,\ldots,
kt\rangle$ be a Toeplitz graph of order $n\ge (2k-1)t+1$. Then,
\begin{align*}
i(G_n)=\prod_{j=1}^kF_{k+1}\left(\left\lfloor\frac{n-j}{t}\right\rfloor+k+2\right).
\end{align*}
\end{theorem}

An independent set of a graph $G$ is a clique in its complement graph $\overline{G}$. Therefore, the number of independent sets in $G$ is the number of cliques in $\overline{G}$.
In addition, it is well-known that the complement of a Toeplitz graph is also a Toeplitz graph.
In the following theorem, we count the number of cliques in a chordal Toeplitz graph to give the number of independent sets in a Toeplitz graph.

\begin{theorem}\label{thm:toecliques} For $k\ge1$, let $G_n=T_n\langle t, 2t,\ldots,
kt\rangle$ be a Toeplitz graph of order $n\ge (2k-1)t+1$. Then, the number of cliques in $G_n$ is\[\left(n-(k-1)t\right)2^k.\]
\end{theorem}

\begin{proof} By Lemmas~\ref{Toeplitz-chordal-2}, $G_n$ has $t$ components $H_1,\ldots,H_t$ and each component is isomorphic to $H_i:=T_{\lfloor(n-i)/t\rfloor+1}\langle 1, 2,\ldots, k\rangle$ for $i=1,\ldots,t$.
Let $\mathcal{K}_i$ be the collection of cliques in $H_i$.
Since $n\ge (2k-1)t+1$, we obtain  $|V(H_i)|=\lfloor(n-i)/t\rfloor+1\ge k+1$.
Now, we partition $\mathcal{K}_i$ into subsets $W_1,\ldots,W_{|V(H_i)|-k}$ where
\[W_1 = \{ K \mid K \subset [k+1]\} \quad \mbox{and} \quad W_j = \{K\cup\{k+j\} \mid K \subset \{j,\ldots,k+j-1\}\}\]
 for $2 \le j \le |V(H_i)|-k$.
 Since any $k+1$
consecutive vertices of $H_i$ form a clique, $|W_1| = 2^{k+1}$ and $|W_j| =
2^k$ for $2 \le j \le |V(H_i)|-k$.
Therefore,
\begin{equation*} |\mathcal{K}_i| = \sum_{j=1}^{|V(H_i)|-k} |W_j| =
\left(|V(H_j)|-k+1\right)2^k.\end{equation*}
Since $\sum_{i=1}^t |V(H_i)|=n$, the number of cliques in $G_n$ is\[
\sum_{i=1}^t|\mathcal{K}_i|=\sum_{i=1}^t(|V(H_i)|-k+1)2^k=\left(n-(k-1)t\right)2^k\]
which completes the proof.\end{proof}

The following result follows immediately from Theorem~\ref{thm:toecliques} and Lemma~\ref{Toeplitz-chordal-1}.
\begin{corollary} Let $G_n$ be a Toeplitz graph with $\overline{G}_n:=T_n\langle t, 2t,\ldots, kt\rangle$ for some positive integer $t$. Then, \[
i(G_n) = \left(n-(k-1)t\right)2^k.\]\end{corollary}

\section{Independent sets of the Bell type Riordan graphs}\label{upperBoundsSec}

In Section~\ref{chordalToeplitzexact}, we introduced the notion of
association of binary words to give a lower bound for the number of
independent sets in a Toeplitz graph. In this section, we continue
to utilize association of binary words to study the number of
independent sets in Riordan graphs of the Bell type.

Let $G$ be a graph with vertices $1,\ldots,n$. We can associate a subset of vertices in a graph $G$ on $n$ vertices with a binary word $x_1x_2\cdots x_n$, where $x_i=1$ if the vertex $i$ is chosen to the subset, and $x_i=0$ otherwise.
We say that a binary word is {\em good} if it corresponds to an independent set in $G$, and the word is {\em bad} otherwise. An independent set of $G$ corresponds to a good word. Thus, the number of good words $x_1x_2\cdots x_n$ equals the number of independent sets in $G$.

A {\em factor} in a word $w=w_1w_2\cdots w_n$ is a subword of the form $w_iw_{i+1}\cdots w_j$. Given two binary sequences $X$ and $W$, we say that $X$ is $W$-{\it avoiding} if $W$ is not a factor of $X$. For example, a binary sequence $1010$ is $11$-avoiding whereas $1100$ is not $11$-avoiding.
It is shown in \cite[Theorem 3.1]{CJKM} that a proper Riordan graph has a Hamiltonian path.
In what follows, we assume that the vertices of a proper Riordan graph $G_n(g,f)$ are labelled by $1,\ldots,n$ so that one of the Hamiltonian paths in $G_n(g,f)$ is $1\rightarrow 2\rightarrow \cdots\rightarrow n$.
Thus, a vertex set $S$ of a proper Riordan graph is an independent set only if the binary sequence associated with $S$ is $11$-avoiding.
Hence, given a proper Riordan graph $G_n(g,f)$, an upper bound for the number of $11$-avoiding good binary words $x_1x_2\cdots x_n$ is an upper bound for the number of independent sets in $G_n(g,f)$.
Thus, in the arguments below, we estimate the number of $11$-avoiding good words.

Recall that the $n$-th Fibonacci number $F(n)$ is defined recursively as follows: $F(0)=F(1)=1$, $F(n)=F(n-1)+F(n-2)$. The following fact is well-known and is easy to prove.

\begin{lemma}\label{fib-lem} The number of $11$-avoiding binary words of length $n\geq 0$ is given by $F(n+1)$.\end{lemma}

 \begin{theorem}\label{fib-initial-up-bound} If a graph $G$ with $n$ vertices has a Hamiltonian path $1\rightarrow  \cdots\rightarrow n$, then $i(G)\leq F(n+1)$. The equality holds for the path graph of length $n-1$.\end{theorem}

 \begin{proof} The first claim follows from Lemma~\ref{fib-lem} because any good word must avoid $11$, or else two adjacent vertices on the Hamiltonian path $1\rightarrow 2\rightarrow \cdots\rightarrow n$ will be chosen to an independent set, which is impossible. The second claim follows from the fact that if $G_n$ is a path graph with length $n-1$, then $G_n$ has no edges apart from those on the Hamiltonian path, and thus any $11$-avoiding binary word corresponds to an independent set in $G_n$. \end{proof}

Since a proper Riordan graph has a Hamiltonian path $1\rightarrow
2\rightarrow \cdots\rightarrow n$, we have the following corollary.
\begin{corollary} For a proper Riordan graph $G_n$ on $n$ vertices, $i(G_n)\leq F(n+1)$ with the equality holding if and only if $G_n = G_n(1,z)$.\label{cor:exact}
\end{corollary}

 In fact, any graph we deal with in this section has a Hamiltonian path, since our concern will be in proper Riordan graphs, so we can assume that any graph on $n$ vertices in this section has the Hamiltonian path $1\rightarrow 2\rightarrow \cdots\rightarrow n$.

We shall give an upper bound for the number of independent sets in
an io-decomposable Riordan graph of the Bell type in terms of the
Pell numbers. The {\em $n$-th Pell number}, denoted by $P_n$, is
defined by \[ P_0 = 0, \ P_1 = 1, \ P_n= 2P_{n-1} + P_{n-2},\] and
it is known that
\begin{align*}
P_n={(1+\sqrt{2})^n-(1-\sqrt{2})^n\over 2\sqrt{2}}
\end{align*}
and the respective generating function is given by
\begin{align}\label{generating-function-lambda}
P(z)=\sum_{n\ge0}P_nz^n={z\over 1-2z-z^2}.
\end{align}
Let $\Delta_n$ and $\tilde\Delta_n$ denote the graphs with $n$
vertices labeled by $1,2,\ldots,n$, and the edge sets given by
\begin{align*}
E(\Delta_n)=\{i(i+1)\;|\;1\le i< n\}\cup\{(2i-1)(2i+1)\;|\;1\le i\le
\lfloor(n-1)/2\rfloor\}
\end{align*}
and
\begin{align*}
E(\tilde\Delta_n)=\{i(i+1)\;|\;1\le i< n\}\cup\{(2i)(2i+2)\;|\;1\le
i\le \lfloor(n-2)/2\rfloor\},
\end{align*}
respectively. One can easily see that
$\Delta_{2n}\cong\tilde\Delta_{2n}$ and
$\Delta_{2n+1}\not\cong\tilde\Delta_{2n+1}$ for $n\ge1$. We denote
$i(\Delta_n)$ and $i(\tilde\Delta_n)$ by $\delta_n$ and
$\tilde\delta_n$, respectively. We set $\delta_0 = \tilde{\delta}_0
=1$.

\begin{lemma}
For $n\ge1$, we have
\begin{align}\label{explicit-delta}
\delta_n=\left\{
\begin{array}{ll}
2P_{(n+1)/2}&
\text{if $n$ is odd,} \\
P_{n/2}+P_{n/2+1} & \text{otherwise;}
\end{array}
\right.
\end{align}
and
\begin{align}\label{explicit-bardelta}
\tilde\delta_n=\left\{
\begin{array}{ll}
P_{(n-1)/2}+2P_{(n+1)/2}&
\text{if $n$ is odd,} \\
P_{n/2}+P_{n/2+1} & \text{otherwise.}
\end{array}
\right.
\end{align}\label{lem:delta}
\end{lemma}
\begin{proof}
Let $B_n$ be the collection of binary words $x_1x_2\cdots x_n$ such
that
\begin{align}\label{string-11-avoiding}
\textrm{both $x_1x_2\cdots x_n$ and $x_1x_3\cdots x_{2\lceil
n/2\rceil-1}$ are 11-avoiding.}
\end{align}
Since $\Delta_n$ and the subgraph of $\Delta_n$ induced by
$\{1,3,\ldots,{2\lceil n/2\rceil-1}\}$ are Hamiltonian, $\delta_n
\le |B_n|$. In addition, since the set of edges of $\Delta_n$ can be
partitioned into the set of edges on the path $1\rightarrow
2\rightarrow \cdots\rightarrow n$ and the set of edges on the path
$1\rightarrow 3\rightarrow \cdots\rightarrow (2\lceil n/2\rceil-1)$,
$|B_n| \le \delta_n$ and so
\begin{equation}  |B_n|=\delta_n.\label{eq:deltab}\end{equation}

By definition, every binary word in $B_{n}$ is of the following form
\begin{align*}
x_1x_2\cdots x_{n-1}0\;\;{\rm or}\;\;x_1x_2\cdots x_{n-2}01.
\end{align*}
\begin{itemize}
\item[($\dag$)] In particular, if $n$ is odd, then every word in $B_{n}$ is of the form
\begin{align*}
x_1x_2\cdots x_{n}0\;\;{\rm or}\;\;x_1x_2\cdots x_{n-3}001
\end{align*}
\end{itemize}
by the second condition in \eqref{string-11-avoiding}.
 Thus, for any positive integer $m$,
\begin{align}\label{recurrence-relation-delta_n}
\delta_{2m}=\delta_{2m-1}+\delta_{2m-2}\;\;{\rm
and}\;\;\delta_{2m+1}=\delta_{2m}+\delta_{2m-2}
\end{align}
with the initial values $\delta_0=1,\delta_1=2$ and $\delta_2=3$. Using
(\ref{recurrence-relation-delta_n}), one can obtain the following
generating function
\begin{align}\label{generating-function-delta_n}
\sum_{n\ge0}\delta_nz^n={1+2z+z^2\over 1-2z^2-z^4}.
\end{align}
Hence, by (\ref{generating-function-lambda}) and
(\ref{generating-function-delta_n}), we obtain
(\ref{explicit-delta}).

Let $\tilde B_n$ be the collection of binary words $x_1x_2\cdots
x_n$ such that
\begin{align}\label{string-11-avoiding-1}
\textrm{both $x_1x_2\cdots x_n$ and $x_2x_4\cdots x_{2\lfloor
n/2\rfloor}$ are 11-avoiding.}
\end{align}
By applying an argument similar to the one used for $B_n$, one can show that
\begin{equation} |\tilde B_n| = \tilde\delta_n \label{eq:deltab2} \end{equation}
whose generating function is
\begin{equation*} \sum_{n\ge0}\tilde\delta_nz^n={1+2z+z^2+z^3\over 1-2z^2-z^4}. \end{equation*}
Therefore, by (\ref{generating-function-lambda}), we obtain (\ref{explicit-bardelta}). Hence the proof is complete. \end{proof}

Let $G_n$ be an io-decomposable Riordan graph of the Bell type with
$n$ vertices where $2^k<n\le2^{k+1}$ for an integer $k\ge2$. Then,
$G_n$ and $\left<V_o\right>$ are proper, so $G_n$ contains the path
$1\rightarrow 2\rightarrow \ldots\rightarrow n$ and the path
$1\rightarrow3\rightarrow \ldots\rightarrow (2\lceil n/2\rceil-1)$
as a subgraph. Thus, $\Delta_n$ is a spanning subgraph of $G_n$ and
so
\begin{align}\label{obvious-upper-bound}
i(G_n)\le \delta_n.
\end{align}
The following theorem improves the upper bound (\ref{obvious-upper-bound}).

\begin{theorem}\label{better-up-bound-io}
Let $G_n$ be an io-decomposable Riordan graph of the Bell type on
$n$ vertices. If $2^k<n\le2^{k+1}$ for an integer $k\ge2$, then
\begin{align}\label{upper-bound-io-decomposable}
i(G_n)\le
\delta_n-(\delta_{2^k-2}-1)\tilde\delta_{n-2^k-3}-\tilde\delta_{n-2^k-1}\sum_{i=1}^{k-1}(\delta_{2^i-2}-1)\tilde\delta_{2^i-3}\alpha_{i+1}
\end{align}
where  $\tilde\delta_m=1$ for $m<0$ and $\alpha_{i+1}=\left\{
\begin{array}{ll}
1&
\text{if $i=k-1$} \\
\prod_{j=i+1}^{k-1}\tilde\delta_{2^j-1} & \text{otherwise.}
\end{array}
\right.$
\end{theorem}
\begin{proof}
  It is known in the paper \cite{CJKM} that, for a positive integer $n$, an integer $k$
satisfying $2^k<n\leq 2^{k+1}$, and an io-decomposable Riordan graph
of the Bell type $G_n$, we obtain the statement that
 \begin{enumerate}
 \item[($\star$)] for a vertex $u \in C:=\{1,2,3,5,9,\ldots,2^k+1\}$ and a  vertex $v \in V(G_n)$ less than $u$, $u$ is adjacent to $v$.
 \end{enumerate}

Let $X=x_1x_2\cdots x_n$ be a binary word satisfying (\ref{string-11-avoiding}).
By ($\star$), $X$ is a bad word if\begin{itemize}
 \item[($\ddag$)] there are at least two indices $i$ and $j$ with $i<j$ such that $x_i=x_j = 1$ and $j\in C$.
 \end{itemize}
Now, we consider the set $A$ of bad words satisfying
(\ref{string-11-avoiding}) and ($\ddag$). For each $i \in C$, let
\begin{equation} A_i = \{ X \in A \mid x_i = 1, x_j = 0 \ \mbox{if}
\ i<j \ \mbox{and} \ j \in C\}.\label{eq:ai}\end{equation} Then,
$A_1$ and $A_2$ are empty sets and $\{A_i \mid i \in C \setminus
\{1,2\}\}$ is a collection of disjoint subsets of $A$.

By $(\dag)$, we have the following statements:
\begin{itemize}
\item a bad word in $A_{2^{k}+1}$ is of the following form
$$
x_1\cdots x_{2^k-2}00x_{2^k+1}00x_{2^k+4}\cdots x_n;
$$

\item a bad word in $A_{2^i+1}$ for some $i=2,3,\ldots,k-1$ is of the following form
$$
x_1\cdots x_{2^i-2}00x_{2^i+1}00x_{2^i+4}\cdots
x_{2^{i+1}}0x_{2^{i+1}+2}\cdots
x_{2^{k}}0x_{2^{k}+2}\cdots x_n.
$$
\end{itemize}
Let $S$ be a subgraph of $\Delta_n$ induced by the vertex set $\{s,\ldots,s+t\}\subseteq V(\Delta_n)$ for some positive integers $s\le n$ and $t\le n-s$. Then $S$ is isomorphic to $\Delta_{t+1}$ and $\tilde\Delta_{t+1}$ if $s$ is odd and even, respectively. Thus, by \eqref{eq:deltab} and \eqref{eq:deltab2}, we obtain the following
statements.
\begin{itemize}
\item[(i)] The number of words $x_1\cdots x_{2^i-2}$ with at least one $1$ satisfying \eqref{string-11-avoiding} is $\delta_{2^i-2}-1$ for each $2 \le i \le k$.
\item[(ii)] The number of words $x_{2^k+4}\cdots x_n$ satisfying \eqref{string-11-avoiding} is $\tilde\delta_{n-2^k-3}$ if $ n-2^k-3\ge 0$.
\item[(iii)] The number of words $x_{2^i+4} \cdots x_{2^{i+1}}$ satisfying \eqref{string-11-avoiding} is $\tilde\delta_{2^i-3}$ for each $2 \le i \le k-1$.
\item[(iv)] The number of words $x_{2^{i+1}+2} \cdots x_{2^{i+2}}$ satisfying \eqref{string-11-avoiding} is $\tilde\delta_{2^{i+1}-1}$ for each $2 \le i < k-1$.
\item[(v)] The number of words $x_{2^k+2} \cdots x_n$ satisfying \eqref{string-11-avoiding} is $\tilde\delta_{n-2^k-1}$.
\end{itemize}
 We let $\tilde\delta_{n-2^k-3} = 1$ if $ 2^k < n < 2^k -3$. Then, by observations (i)--(v), we have \[|A_{2^k+1}| = (\delta_{2^k-2}-1)\tilde\delta_{n-2^k-3}, \quad |A_{2^{k-1}+1}| =  \tilde\delta_{n-2^k-1} (\delta_{2^k-2}-1)\tilde\delta_{2^{k-1}-3},\] and\[
 |A_{2^i+1}| = \tilde\delta_{n-2^k-1}(\delta_{2^i-2}-1)\tilde\delta_{2^i-3}\prod_{j=i+1}^{k-1}\tilde\delta_{2^j-1}\] for $2 \le i < k-1$.
Hence we obtain the desired result.
\end{proof}

\begin{remark} From the way we obtain inequality~\eqref{upper-bound-io-decomposable}, we can see
the equality holds in \eqref{upper-bound-io-decomposable} if $G_n$ is a graph with $n$ vertices labeled as $1,2,\ldots,n$ and the edge set of $G_n$ is given by
$$E(G_n)=E(\Delta_n)\cup\{ij\;|\;1\le i<j\;{\rm and}\; j=2,3,5,\ldots,2^k+1 \}$$
for the integer $k$ satisfying $2^k< n\le 2^{k+1}$. We have checked that the graph $G_n$ is a Catalan graph $CG_n$ for $1\le n\le 5$, but it is not a Riordan graph for $n\ge 6$.
\end{remark}

The Pascal graph $PG_n=G_n\left(\frac{1}{1-z}, \frac{z}{1-z}\right)$ is an io-decomposable Riordan graph of the Bell type and so it satisfies the upper bound given in Theorem~\ref{better-up-bound-io}. As a matter of fact, an upper bound for the number of independent sets in $PG_n$ can be improved by utilizing some properties of $PG_n$.

\begin{theorem}\label{better-up-bound-Pascal-1} Let $2^k < n \leq 2^{k+1}$ and $k \geq 2$.
Then,
\begin{align*}
i(PG_n)\le &\delta_n+1+2^{\lfloor
n/2\rfloor-1}-(\delta_{2^k-2}-1)\tilde\delta_{n-2^k-3}\\
&-\tilde\delta_{n-2^k-1}\left(2\prod_{i=1}^{k-1}\tilde\delta_{2^i-1}+\sum_{i=1}^{k-1}(\delta_{2^i-2}-1)\tilde\delta_{2^i-3}\alpha_{i+1}\right)
\end{align*}
where  $\tilde\delta_m=1$ for $m<0$ and $\alpha_{i+1}=\left\{
\begin{array}{ll}
1&
\text{if $i=k-1$} \\
\prod_{j=i+1}^{k-1}\tilde\delta_{2^j-1} & \text{otherwise.}
\end{array}
\right.$
\end{theorem}
\begin{proof}
  By the definition of $PG_n$, $ij\in E(PG_n)$ if and only if
$[z^{i-2}]{z^{j-1}\over (1-z)^j}\equiv 1 \pmod{2}$ where $n\ge
i>j\ge1$. By substituting $j=1$ and $j=2$, we obtain the
following two facts:
\begin{itemize}
\item[(P1)] the vertex $1$ in $PG_n$ is adjacent to all other vertices;
\item[(P2)] the vertex $2$ in $PG_n$ is adjacent to all odd  vertices for $n\ge
2$.
\end{itemize}

Let $C= \{1,2,3,5,9,\ldots,2^k+1\}$ and $B'_n=B_n \setminus \left(\cup_{i\in C} A_i \right)$ where $B_n$ was defined in the proof of Lemma~\ref{lem:delta} and $A_i$ was defined in \eqref{eq:ai}.
Consider the following two sets of words $X:=x_1x_2\cdots x_n$:
\[
 A'_1:=  \{ X \in B_n \mid x_1 = 1, x_j = 0\ \mbox{for each} \ j \in C \setminus\{1\}, x_l = 1\ \mbox{for some} \ l \notin C\};\]
 \[
A'_2:= \{ X \in B_n  \mid x_2 =1, x_j = 0\ \mbox{for each} \ j \in C \setminus\{2\}, x_l = 1\ \mbox{for some odd}\ l \notin C\}.\]
Since $X\not \in \cup_{i\in C} A_i$ if $X\in A_1'\cup A_2'$,  $A'_1 \cup A'_2$ is a subset of $B'_n$.

By (P1) and (P2), we can see that $A'_1$ and $A'_2$ are sets of bad words associated with $PG_n$.
Now, we count the words in each of $A'_1$ and $A'_2$, and subtract them from $|B'_n|$ to improve the upper bound given in Theorem~\ref{better-up-bound-io}.

 We count the number of words in $A'_j$ for each $j=1,2$. Fix $j\in \{1,2\}$. Take a word $X=x_1x_2\cdots x_n$ in $A'_j$. Then,  by the definition of $A'_j$,  $x_{2^{i}+1} = 0$ for each $1 \le i \le k$; $x_1=1$ and $x_2=x_3=0$ if $j=1$ and $x_1=x_3=0$ and $x_2=1$ if $j=2$.
 Therefore, in order to count the words in $A'_j$, it is sufficient to count the number of ways to determine the factors \[ x_4, x_6x_7x_8, \ldots,
x_{2^{k-1}+2}x_{2^{k-1}+3}\cdots x_{2^k}, x_{2^k+2}x_{2^k+3}\cdots
x_n \] so that $X$ satisfies \eqref{string-11-avoiding}. Take $i \in
\{1,\ldots,k-1\}$. Since $X \in B_n$, $X$ satisfies
\eqref{string-11-avoiding}. Thus, the factor
$x_{2^i+2}x_{2^i+3}\cdots x_{2^{i+1}}$ satisfies
\eqref{string-11-avoiding-1}. Therefore, $X \in A'_j$ if and only if
the factor $x_{2^i+2}x_{2^i+3}\cdots x_{2^{i+1}}$ satisfies  for
each $i \in \{1,\ldots,k-1\}$; $X \neq 100\ldots 0$ for $j=1$, and
not all odd entries of $X$ are $0$ for $j=2$.

 For each $i \in \{1,\ldots,k-1\}$, there are $\tilde\delta_{2^i-1}$   $x_{2^i+2}x_{2^i+3}\cdots x_{2^{i+1}}$ factors satisfying \eqref{string-11-avoiding-1}.
 In addition, there are  $\tilde\delta_{n-2^k-1}$ $x_{2^k+2}x_{2^k+3}\cdots x_n$ factors satisfying \eqref{string-11-avoiding-1}. Thus,
\begin{align}\label{the number of bad words P1}
|A'_1| = \tilde\delta_{n-2^k-1}\prod_{i=1}^{k-1}\tilde\delta_{2^i-1} -1.
\end{align}
In particular, if $j=2$, there are $2^{\left\lfloor \frac{n}{2}\right\rfloor-1}$ words in $B'_n$ such that $x_i=1$ only if $i$ is even. Thus,
\begin{align}\label{odd-univ-2}
|A'_2| =\tilde\delta_{n-2^k-1}\prod_{i=1}^{k-1}\tilde\delta_{2^i-1}\ -\ 2^{\left\lfloor \frac{n}{2}\right\rfloor-1}.
\end{align}
By~\eqref{the number of bad words P1} and~\eqref{odd-univ-2}, we obtain the desired result.
\end{proof}

The tools which we have developed to utilize association of binary
words can be used in finding the independence number and the number
of {\it maximal} independent sets in an io-decomposable Riordan
graph of the Bell type.

\begin{theorem}\label{maximal-independent-sets-io-decom}
The independence number of an io-decomposable Riordan graphs is
$\lfloor n/2\rfloor$. In particular, the number of maximal
independent sets in an io-decomposable Riordan graph is at most $2$
if $n$ is even and at most $4$ if $n$ is odd.
\end{theorem}
\begin{proof} Let $G_n$ be an io-decomposable Riordan graph and $\alpha(G_n)$ denote the independence number of $G_n$. First we
show that $\alpha(G_n)=\lfloor n/2\rfloor$. Since $\left<V_e\right>$
is a null graph, $V_e$ is an independent set. Thus, $\alpha(G_n)\ge
|V_e|=\lfloor n/2\rfloor$.
  Since $G_n$ is proper, a good binary word $x_1x_2\cdots x_n$ associated with $G_n$ is $11$-avoiding.
  It leads to $\alpha(G_n)\le\lfloor n/2\rfloor$. Thus, $\alpha(G_n)=\lfloor n/2\rfloor$.

Let $d$ be an odd integer with $3\le d<n$. Since $\langle V_o \rangle$ is proper, a subword $x_1x_3\cdots x_{2\lceil n/2\rceil-1}$ of a good binary word $x_1x_2\cdots x_n$ associated with $G_n$ is $11$-avoiding as well.
Therefore, if an independent set $I$ contains a vertex $d$, then $I$ contains none of $d-2$, $d-1$, $d+1$ and so $|I| \le \frac{n-3}{2} <\lfloor n/2\rfloor$.
 Thus, there is no maximal independent set containing an odd integer in between $1$ and $n$.
  Hence, the possible maximal independent sets are given by the following:
\begin{itemize}
\item[(i)] $V_e$ and $(V_e\setminus \{2\})\cup\{1\}$ if $n$ is even;
\item[(ii)] $V_e$, $(V_e\setminus \{2\})\cup\{1\}$, $(V_e\setminus
\{n-1\})\cup\{n\}$ and  $(V_e\setminus \{2,n-1\})\cup\{1,n\}$ if
$n\ge5$ is odd,
\end{itemize}
which implies the desired result.
\end{proof}

\begin{corollary}
The Pascal graph $PG_n$ has a unique maximal independent set $V_e$
if $n$ is an even integer greater than $2$ or $n=2^k+1$ for some integer $k\ge 2$.
\end{corollary}
\begin{proof} By Theorem~\ref{maximal-independent-sets-io-decom}, $\alpha(PG_n) = \lfloor \frac{n}{2} \rfloor$.
Since $\lfloor\frac{n}{2} \rfloor >1$ for $n \ge 4$, a maximal independent set of $PG_n$ has the size greater than $1$ if $n \ge 4$.

Assume that $n$ is an even integer greater than $2$ or $n=2^k+1$ for some integer $k \ge 2$.
By the definition of Pascal graph, the vertex $1$ is adjacent to all the other vertices in $PG_n$ for all $n\ge2$.
Therefore, an independent set which contains $1$ cannot have any other elements.
If $n$ is an even integer greater than $2$, then  a maximal independent set of $PG_n$ does not contain $1$ for $n \ge 4$ and so $V_e$ is a unique maximal independent set of $PG_n$ by  (i) in the proof of Theorem \ref{maximal-independent-sets-io-decom}.

Now, we assume that $n=2^k+1$ for some integer $k \ge 2$. It is known \cite{CJKM} that the vertex $2^k+1$ in
an io-decomposable Riordan graph $G_n(g,zg)$ is adjacent to all the other vertices if $n=2^k+1$.
Therefore, an independent set which contains $2^k+1$ cannot have any other elements.
 Since $k \ge 2$, $2^k+1 \ge 4$ and so a maximal independent set of $PG_n$ contains none of $1$ and $2^k+1$. Thus, $V_e$ is a unique maximal independent set of $PG_n$ by  (ii) in the proof of Theorem~\ref{maximal-independent-sets-io-decom}.\end{proof}

Next, we give a lower bound for the number of independent sets in a graph through a decomposition of graphs.
 For a $(0,1)$-matrix $M$, we denote by $\sigma_0(M)$ and $\sigma_1(M)$ the number of $0$'s in $M$ and the number of $1$'s in $M$, respectively.
 \begin{theorem}\label{Rio1}
 Let $G$ be a graph with $n$ vertices labeled by $1,2,\ldots,n$. Then \[i(G)\ge i(\langle V_o\rangle)+i(\langle V_e\rangle)+\lceil n/2\rceil\ \lfloor n/2\rfloor-|E(G)|+|E(\langle V_o\rangle)|+|E(\langle V_e\rangle)|.\]
 \end{theorem}
 \begin{proof}
  Let $A$ be the adjacency matrix of $G$ such that the rows and columns $1,2,\ldots,n$ correspond to the vertices labeled in the order $1,3, \ldots, 2\lceil n/2\rceil-1, 2, 4,
\ldots, 2\lfloor n/2\rfloor$. Then,
 $$A=\left[
     \begin{array}{cc}
       X & B \\
       B^T & Y \\
     \end{array}
   \right]
 $$
where $X$ is the adjacency matrix of  $\langle V_o\rangle$ and $Y$ is the adjacency matrix of $\langle V_e\rangle$.
We consider three types of independent sets in $G$ as follows: the independent sets in $\langle V_o \rangle$; the independent sets in $\langle V_e \rangle$; the independent sets of size $2$ formed by an element in $V_o$ and an element in $V_e$. Note that an independent set of the third type corresponds to a $0$ in $B$. Therefore,
\begin{equation} \label{ah4}
 i(G_n)\ge i(\langle  V_o\rangle)+i(\langle  V_e\rangle)+\sigma_0(B)
 \end{equation}
  Since an edge between $V_o$ and $V_e$ corresponds to a $1$ in $B$, we have
    \[\sigma_1(B)=|E(G)|-|E(\langle V_o\rangle)|-|E(\langle V_e\rangle)|.\] Since $\sigma_0(B)+\sigma_1(B)=\lceil n/2\rceil \lfloor  n/2\rfloor$,
  \[\sigma_0(B)=\lceil
n/2\rceil\ \lfloor n/2\rfloor-|E(G)|+|E(\langle V_o\rangle)|+|E(\langle V_e\rangle)|.\]
By substituting this into (\ref{ah4}), we obtain the desired result. \end{proof}

The following result is a corollary of Theorem~\ref{Rio1}.
 \begin{corollary}\label{cor-io-dec}
 Let $G_n:=G_n(g,f)$ be an io-decomposable Riordan graph with $n\ge2$ vertices. Then, $$i(G_n)\ge i(G_{\lceil n/2\rceil})+2^{\lfloor n/2\rfloor}+\lceil n/2\rceil\,\lfloor n/2\rfloor-|E(G_n)|+|E(G_{\lceil n/2\rceil})|.$$
 \end{corollary}
\begin{proof} By the definition of an io-decomposable Riordan graph with $n$ vertices labeled by $1,\ldots,n$, $\langle V_o\rangle \cong G_{\lceil n/2\rceil}$ and $E(\langle V_e\rangle) = \emptyset$, so the result follows.\end{proof}

In what follows, we give a lower bound for the number of independent sets in io-decomposable Riordan graphs of the Bell type. We first present two properties of io-decomposable Riordan graphs of the Bell type.

\begin{lemma}[\cite{CJKM}]\label{multipartite-graph-io} An io-decomposable Riordan graph $G_{n}(g,zg)$ is $(\lceil\log_2n\rceil+1)$-partite with the partitions
$V_1,V_2,\ldots,V_{\lceil\log_2n\rceil+1}$ where
\begin{align*}
V_j=\left\{2^{j-1}+1+(i-1)2^j\;|\;1\le i\le\left\lfloor
{n-1+2^{j-1}\over 2^j}\right\rfloor\right\}\end{align*}
for $1\le j\le\lceil\log_2n\rceil$ and $V_{\lceil\log_2n\rceil+1}=\{1\}$.
\end{lemma}

\begin{lemma}[\cite{CJKM}]\label{Riordan-graph-Bell-type-decomposition}
Let $G_n(g,zg)$ be an io-decomposable Riordan graph. Then, the lines of the adjacency
matrix of $G_n(g,zg)$ can be simultaneously permuted to have the matrix
\begin{eqnarray*}
\left(
\begin{array}{cc}
X & B \\
B^{T} & O%
\end{array}
\right)
\end{eqnarray*}
where $X$ is the adjacency matrix of $\left< V_{o}\right>\cong
G_{\lceil n/2\rceil}(g(z),zg(z))$ and
\begin{eqnarray}\label{e:bm1}
B\equiv(zg,zg)_{\lceil n/2\rceil ,\lfloor
n/2\rfloor}+((zg)^{\prime}(\sqrt{z}),zg)) _{\lfloor n/2\rfloor
,\lceil n/2\rceil }^{T}\pmod 2.
\end{eqnarray}
\end{lemma}

 \begin{theorem}\label{yet-better-up-bound-io} Let $G_n$ be an io-decomposable Riordan graph of the Bell type with $n\ge2$ vertices.
 Then, $$i(G_n)\geq 2-\lceil\log_2n\rceil+\sum_{j=1}^{\lceil\log_2n\rceil}\left(2^{\alpha_{j}}+{\alpha_{j+1}^2-\alpha_{j+1}\over2}\right)$$
 where $\alpha_j=\left\lfloor{n-1+2^{j-1}\over2^j}\right\rfloor$.
 \end{theorem}
\begin{proof} For each $1\le j\le\lceil\log_2n\rceil +1$, the subgraph $G_n[V_j]$ induced by the vertex subset $V_j$ defined in Lemma \ref{multipartite-graph-io} is a null graph. In addition, $|V_j| = \alpha_j$ for each $1\le j\le\lceil\log_2n\rceil $ and $|V_{\lceil\log_2n\rceil}+1|=1$. Thus
\begin{align}\label{equation-1}
i(G_n[V_{\lceil\log_2n\rceil+1}])=2\; {\rm and}\;
i(G_n[V_j])=2^{\alpha_j}\;
\end{align}
for each $j=1,\ldots,\lceil\log_2n\rceil$.

 Now, we count the $2$-element independent sets each of which is formed by a vertex in $V_i$ and a vertex in $V_j$ for some $1 \le i < j \le \lceil\log_2n\rceil$.
By definition, $V_1 = \{ 2i \mid 1 \le i \le \left\lfloor \frac{n}{2} \right\rfloor\} = V_e$.
Therefore, $\bigcup_{i=2}^{\lceil \log_2n\rceil+1}=V_o$ and so the number of independent sets $\{u,v\}$ for $u \in V_1$ and $v \in V_j$, $j > 1$ is the number of zeros in the matrix $B$ given in Lemma~\ref{Riordan-graph-Bell-type-decomposition}.
 By (\ref{e:bm1}), the number of zeros below the main diagonal in $B$ is equals to the number of
zeros below the main diagonal in $A(\left<V_o\right>)$. Recall that
$\left<V_2\right>$ is the null graph of order $\alpha_2$. Since
$V_2$ is a subgraph of $V_o$, there are at least
$(\alpha_2^2-\alpha_2)/2$ zeros below the main diagonal in
$A(\left<V_o\right>)$. Therefore, there are at least
$(\alpha_2^2-\alpha_2)/2$ independent sets $\{u,v\}$ for $u\in V_1$
and $v\in V(G_n)\backslash V_1$. Similarly, we can show that there
are at least $(\alpha_{j+1}^2-\alpha_{j+1})/2$ independent sets
$\{u,v\}$ for $u\in V_j$ and $v\in \cup_{ j < i \le n} V_i$ for each
$2 \le j \le \lceil\log_2n\rceil$. Hence, we have
$$i(G_n)\geq 2-\lceil\log_2n\rceil+\sum_{j=1}^{\lceil\log_2n\rceil}\left(2^{\alpha_{j}}+{\alpha_{j+1}^2-\alpha_{j+1}\over2}\right)$$
where $\lceil\log_2n\rceil$ is the number of the empty sets which overlapped in (\ref{equation-1}). \end{proof}

\section{Directions of further research}\label{further-research-sec}
This paper focuses on giving lower and upper bounds for the number of independent sets for various classes of Riordan graphs. Of course, the most challenging thing here is in finding exact enumeration in question, which does not seem to be feasible in the context due to the problem generality. In any case, there are other questions one can ask. For example, the Catalan graphs are io-decomposable of the Bell type, so the results of Theorem~\ref{better-up-bound-io} can be applied to them. However, can we provide a more accurate upper bound, and some lower bound for this class of graphs?

In addition, Table~\ref{tab-data} gives initial values for the number of independent sets for the Fibonacci graphs and Motzkin graphs. However, obtaining any lower/upper bounds, or exact enumeration, for the number of independent sets for these graphs in general remains  an open problem.

 \end{document}